\documentclass{amsart}
\usepackage{amssymb,latexsym}
\newtheorem{theorem}{Theorem}[section]

\newtheorem{definition}[theorem]{Definition}

\newtheorem{lemma}[theorem]{Lemma}

\def\F{\mathbb{F}}

\title{Fundamental Groups of Random Clique Complexes}

\date{\today}

\begin{document}
\begin{abstract}
Clique complexes of Erd\H{o}s-R\'{e}nyi random graphs with edge probability between $n^{-{1\over 3}}$ and $n^{-{1\over 2}}$ are shown to be aas not simply connected.  This entails showing that a connected two dimensional simplicial complex for which every subcomplex has fewer than three times as many edges as vertices must have the homotopy type of a wedge of circles, two spheres and real projective planes. Note that $n^{-{1\over 3}}$ is a threshold for simple connectivity and $n^{-{1\over 2}}$ is one for vanishing first $\F_2$ homology.  
\end{abstract}
\maketitle
\section{Introduction}
If $n$ is a positive integer and $p\in[0,1]$ is a probability write $K(n,p)$ for the probability measure on 2-dimensional simplicial complexes obtained by taking vertex set $[n]=\{1,\ldots ,n\}$ and edges chosen from all ${n\choose 2}$ possibilities independently each with probability $p$ and all triangles for which all three edges were chosen.  This is the 2-skeleton of the clique complex of the Erd\H{o}s-R\'{e}nyi random graph.  Write aas for asymptotically almost surely where the limit involved is $\lim_{n\rightarrow\infty}$.  

\begin{theorem}
For any $\epsilon>0$ and $n^{\epsilon-{1\over 2}}\leq p_n\leq n^{-\epsilon-{1\over 3}}$ the group $\pi_1(K(n,p_n))$ is aas hyperbolic and nontrivial.  
\end{theorem}
This is proven largely by following the notation and blueprint in BHK ([1]).  The main difference here is 
\begin{theorem}
If $X$ is a finite connected two dimensional simplicial complex for which every subcomplex $Y$ has ${f^0Y\over f^1Y}>{1\over 3}$ then $X$ has the homotopy type of a wedge of circles, two spheres and real projective planes and contains a subcomplex with the homotopy type of a wedge of circles and real projective planes for which the inclusion induces an isomorphism of fundamental groups.  
\end{theorem}
Here $f^iY$ is the number of $i$-dimensional faces in $Y$.  This  is a corollary of theorem 2.1 and replaces BHK Lemma 4.1 in which ${f^0Y\over f^1Y}>{1\over 3}$ is replaced by ${f^0Y\over f^2Y}>{1\over 2}$.

Note that if $p_n\leq n^{-1-\epsilon}$ then $K(n,p_n)$ is aas a disconnected forest and if $n^{-1+\epsilon}\leq p_n\leq n^{-{1\over 2}-\epsilon}$ then by Lemma 3.8 $K(n,p_n)$ is aas connected and collapsible to a graph with cycles.  If $n^{-{1\over 3}+\epsilon}\leq p_n$ then $K(n,p_n)$ is aas simply connected from Kahle's [3] Theorem 3.4.

\section{Definitions}
Recall webs from BHK Definition 4.5, $L$ from Definition 4.6 and modify Definition 4.7 to call a web $W$ $k$-admissible if every $Y\subseteq W$ has $(L+k\chi)Y>0$.  
Note that BHK studies $2$-admissible webs and this note studies $3$-admissible ones.  

\begin{theorem}(Related to Lemma 4.16 of BHK)
If $W$ is a connected $3$-admissible $2$-dimensional web with $g(W)\geq 3$ then $|X|$ has the homotopy type of a wedge of circles, two spheres and real projective planes.  
\end{theorem}
\begin{proof}
\begin{lemma} If the theorem fails then there is a counterexample $W$ with $\delta(W)\geq 2$.
\end{lemma}
\begin{proof}
Recall from BHK the Definition 4.14 of $K(W)$ which has $\delta(K(W))\geq 2$.  $K(W)$ is also a counterexample no larger than $W$.  
\end{proof}

As in the proof of Lemma 4.22 in BHK write the admissibility sum locally as $(L+3\chi)W=\sum_v K_v +\sum_c K_c +\sum_m K_m$ where the sums are over faces of $W$ with empty boundary and dimensions zero, one and two respectively.  

\begin{definition}
Partially order connected two dimensional webs with $\delta\geq 2$ by setting $W\leq W'$ if $$(f^2W, -(L+3\chi)W, f^1W, f^0W, -K_{v_1},\ldots -K_{v_{f^0W}})$$ $$\leq(f^2W', -(L+3\chi)W', f^1W', f^0W', -K_{v'_1},\ldots -K_{v'_{f^0W'}})$$ lexicographically where $v_1\ldots$ are the vertices of $W$ with $K_{v_i}\geq K_{v_{i+1}}$ and similarly for $W'$.  
\end{definition}

Note that if there is a counterexample to the theorem in this partial order there is also a minimal one. 
Choose $W$ to be a minimal counterexample. 

\begin{definition}
Call a two dimensional complex normal if all vertex links are connected graphs and 2-normal if all vertex links are 2-connected graphs.
\end{definition}
\begin{lemma}(Related to BHK 4.18) If $W$ is a minimal counterexample then $\delta(W)\geq 3$ and $W$ is $2$-normal.  
\end{lemma}
\begin{proof}
If $W$ is not normal choose $N_v(W)$ a disconnected link at $v$ of $W$ and consider 
$f:W'\rightarrow W$ the normalization map and note that every subweb $Y'\subseteq W'$ has $LfY'=LY'$ and $\chi fY' \leq\chi Y'$ so that $W'$ is also admissible and $W'<W$ in the above order since $f^2W'=f^2W$, $LW'=LW$ and $\chi W'=\chi W-1$.  Thus some component of $W'$ is a smaller counterexample since $|W|$ is the wedge of the components of $|W'|$ and some possibly circles.  

If $W$ is not 2-normal choose $N_v(W)$ with a cut point and consider the zipping map $z:W'\rightarrow W$ so that $W'$ has one more vertex and one more edge than $W$.  Note that for any $Y'\subseteq W'$ there is $LY'\geq LzY'$ and $\chi Y'\geq \chi zY'$ so that $W'$ is also admissible and $W'<W$ in the above partial order since $f^2W'=f^2W$, $\chi W'=\chi W$, $LW'=LW+2\mu e$.  Since $|z|$ is a homotopy equivalence this makes $W'$ a smaller counterexample.  

Recall from BHK Lemma 4.11 the Definition of $C(W)$, which is defined for $2$-normal webs is again a counterexample and $W\leq CW$, so $W=CW$ and hence $\delta(W)\geq 3$.  
\end{proof}

\begin{lemma}(Related to BHK 4.20) If $W$ is a minimal counterexample then $W$ has no monogons or digons.  
\end{lemma}
\begin{proof}
If $F$ is a digon in $W$ with edges $e$ and $f$ having $\mu e \leq \mu f$ consider the 
collapse $W'$ to the shorter edge $e$ and the collapse map $\phi:W\rightarrow W'$.  
Note that for every subweb $Y'\subseteq W'$ there is $Y=\phi^{-1}Y'$ and sometimes 
$Y_e=Y-\{e,F\}$ or $Y_f=Y-\{f,F\}$ are also subwebs.  Each of these has $\chi Y_{(?)}=\chi Y'$ and at least one of $LY_{(?)}\leq LY'$ so that $W'$ is also admissible and $W'<W$ in the above order since $f^2W'=f^2W-1$.   

If $F$ is a monogon in $W$ with edge $e$ consider the collapse $\phi:W\rightarrow W'$ of $F$ to a point.  
Note that for every subweb $Y'\subseteq W'$ there is $Y=\phi^{-1}Y'$ and sometimes 
$Y_e=Y-\{e,F\}$ is also a subweb.  Both of these has $\chi Y_{(e)}=\chi Y'$ and at least one of $LY_{(e)}\leq LY'$ so that $W'$ is also admissible and $W'<W$ in the above order since $f^2W'=f^2W-1$.   
\end{proof}

\begin{lemma}
If $W$ is a minimal counterexample with $v$ a vertex, $e$ and $f$ edges containing $v$, $c$ a circular 1-face, $F$ a 2-face containing $e$ and $f$ and $G$ a 2-face containing $c$ then each of the following variables is a non negative integer:
\newline $\hat{f^1}v=f^1v-3$, 
\newline $\hat{f^2}e=f^2e-3$, 
\newline $\hat{\mu}e=\mu e-1$, 
\newline $\hat{\chi}F=-\chi F+1$, 
\newline $\hat{a}(v,e,f,F)=\hat{\mu} e + \hat{\mu} f$,
\newline $\hat{m}(v,e,f,F)=\mu\partial F-\hat{a}(v,e,f,F)-3$, 
\newline $\hat{\mu\partial}(c,G)=\mu\partial G-f^Gc\mu c$ and
\newline $\hat{f^G}c=f^Gc-1$.
\end{lemma}
Here $f^im$ is the number of $i$-dimensional faces containing the face $m$ and $f^Gc$ is the degree of the map from the boundary of $G$ to $c$.  See BHK.  

Note that if $v$ is a vertex of $W$ then using $$\sum_{\{e|v\in e\}}\hat{\mu}e=-\sum_{\{e|v\in e\}}{1\over 3}\hat{\mu}e\hat{f^2}e + \sum_{\{e,f,F|v\in e\in F,v\in f\in F, e\not= f\}}{1\over 3}\hat{a}(v,e,f,F)$$ yields
$$K_v$$ 
$$= 3-{3\over 2}f^1v+\sum_{\{(e,f,F)|v\in e,f\in F\}}{3\chi F {1\over 2}(\mu e+\mu f)\over\mu\partial F} + \sum_{\{e|v\in e\}}\mu e -\sum_{\{(e,f,F)|v\in e,f\in F\}}{1\over 2}(\mu e+\mu f)$$
$$= {3\over 2} -{1\over 2} \hat{f^1}v - \sum_{\{e|v\in e\}}{1\over 3}\hat{\mu}e\hat{f^2}e $$
$$- \sum_{\{(e,f,F)|v\in e,f\in F\}}{3\hat{\chi}F + \hat{m}(v,e,f,F) +\hat{a}(v,e,f,F)[{3\over 2}\hat{\chi}F + {1\over 6}\hat{m}(v,e,f,F) + {1\over 6}\hat{a}(v,e,f,F)]\over 3 + \hat{m}(v,e,f,F) + \hat{a}(v,e,f,F)}$$
$$={3\over 2} -{1\over 2} \hat{f^1} - \sum_{\{e|v\in e\}}{1\over 3}\hat{\mu}\hat{f^2} - \sum_{\{(e,f,F)|v\in e,f\in F\}}{3\hat{\chi} + \hat{m} +\hat{a}[{3\over 2}\hat{\chi} + {1\over 6}\hat{m} + {1\over 6}\hat{a}]\over 3 + \hat{m} + \hat{a}}.$$

Similarly, if $c$ is a circular one dimensional face of $W$ then $$K_c=\mu c\left[2-\sum_{\{G_c\in G\}}f^Gc{(\hat{f^G}c+\hat{\chi}G)\mu c+\hat{\mu\partial}(c,G)\over f^Gc\mu+\hat{\mu\partial}(c,G)}\right].$$   

Finally if $m$ is two dimensional with empty boundary then $$K_m=3\chi(m).$$  

Since $(L+3\chi)W>0$ there is some face $F$ with empty boundary and $K_F>0$.  

If $m$ is $2$ dimensional, without boundary and $K_m>0$  
then $\chi(W)>0$ so $|W|$ is a sphere or projective plane.  
\begin{lemma}
If $W$ is a minimal counterexample and $c$ is a circular face then $K_c\leq 0$.
\end{lemma}
\begin{proof}Assume $K_c>0$.  If $G$ contains $c$ and $\hat{\mu\partial}(c,G)=0$ then the contribution of $G$ to 
${K_c\over \mu c}$ is $-\hat{f^G}c-\hat{\chi}G$ so only twice wrapped disks ($\hat{f}=1$, $\hat{\chi}=0$) cross caps ($\hat{f}=0$, $\hat{\chi}=1$) and singly wrapped disks ($\hat{f}=0$, $\hat{\chi}=0$) can occur if $K_c$ is to be positive.  
If $\hat{\mu\partial}(c,G)\not=0$ then $\hat{\mu\partial}(c,G)\geq g(W)\geq 3$ and $\hat{\chi}G\geq 1$ so that $\hat{f^G}c=0$.  
The only faces which do not subtract at least one are the singly wrapped disks but if $W$ is a minimal counterexample and $c$ a circular face there is at most one of these.  

This leaves only the case of one doubly wrapped and one singly wrapped disk, which has the homotopy type of a sphere and is therefore not a counterexample.    
\end{proof}
\begin{lemma}
If $W$ is a minimal counterexample and $v$ is a vertex then $K_v\leq 0$.
\end{lemma}
\begin{proof}Assume that $v$ is a vertex and $K_v>0$.  

\begin{lemma}(only long double edges in links)
If $W$ is a minimal counterexample, $K_v\geq K_u$ for every $u$ adjacent to $v$ and there are edges $e$ and $f$ and $2$-faces $F$ and $G$ with $F$ and $G$ forming a double edge connecting $e$ and $f$ in the link of $v$ then $\mu\partial F> 2(\mu e+\mu f)$ (or equivalently $\hat{m}F> \hat{\mu}e +\hat{\mu}f+1$).
\end{lemma}

Note that this implies that every double edge subtracts at least ${4\over 5}$ from $K_v$.

\begin{proof}
Assume not and consider $j:W''\rightarrow W$ the deletion of $G$ and $i:W''\rightarrow W'$ the addition of $G'$ which slides $G$ across $F$.  Note that $|W|$ and $|W'|$ are homotopy equivalent and if $G'\in Y'\subseteq W'$ then $(L+3\chi) Y'\geq (L+3\chi) Y$ for either $Y=Y'-G'+G$ or 
$Y=Y'-G'+G+F$ so that $W'$ is admissible.  The former works if $e$ and $f$ are in $Y'$, in which case $\chi Y'=\chi Y$ and $LY'=LY-\mu\partial F+2\mu e+2\mu f\geq LY$.  Otherwise the latter works, with four cases depending on the intersection of $Y'$ with $e$ and $f$.  If the intersection is empty then  $LY'=LY$ and $\chi Y'=\chi Y$.  If the intersection is only $v$ then $LY'=LY$ and $\chi Y'=\chi Y+1$.  If the intersection is an edge (wlog $e$) then $LY'=LY+2\mu e$ and $\chi Y'=\chi Y$.  Also $X'<X$ in the above order since $f^2X'=f^2X$, $\chi X'=\chi X$, $LX'=LX+\mu\partial F-2\mu e-2\mu f\leq LX$ and $K_{v'}>K_v$.  
\end{proof}
\begin{lemma}(only long triangles in links) If $W$ is a minimal counterexample, $K_v\geq K_u$ for every $u$ adjacent to $v$ and there are edges $e$, $f$ and $g$ and $2$-faces $E$, $F$ and $G$ with $E$, $F$ and $G$ forming the edges and $e$, $f$ and $g$ the vertices of a triangle in the link of $v$ then $mF+mE-2\mu g> 2(\mu e +\mu f)$ (or equivalently $\hat{m}E+\hat{m}F> \hat{\mu}e+\hat{\mu}f+2$).
\end{lemma}

Note that this implies that every triangle in the link of $v$ subtracts at least ${3\over 4}$ from $K_v$ and every square with diagonal subtracts at least ${6 \over 5}$.

\begin{proof}
Assume not and consider $j:X''\rightarrow X$ the deletion of $G$ and $i:X''\rightarrow X'$ the addition of $G'$ which slides $G$ across $E$ and $F$.  Note that $|X|$ and $|X'|$ are homotopy equivalent and if $G'\in Y'\subseteq X'$ then $(L+3\chi)Y'\geq (L+3\chi)Y$ for $Y=Y'-G'+G$ or $Y=Y'-G'+G+F+E$.  Also $X'<X$ in  the above order since $f^2X'=f^2X$, $\chi X'=\chi X$, $LX'\geq LX$ and $K_{v'}>K_v$.  
\end{proof}

A case analysis now eliminates any minimal counterexample, proving Lemma 2.9.  
\end{proof}

This completes the proof of Theorem 2.1. 
\end{proof}

\smallskip\noindent{\it Proof of Theorem 1.2:}
The first part follows from Theorem 2.1.
Since Theorem 2.1 also holds for subcomplexes the argument in the proof of Theorem 4.1 in BHK completes the proof.  
\hfill$\square$\medskip

\section{fundamental groups}
The fundamental group restriction is much like in BHK.  
\begin{definition}
If $X$ is a 2-dimensional connected simplicial complex then 
$$ e^0_1X=\hbox{min}_{Y\subseteq X}{f^0Y\over f^1Y}$$
if also $X$ contains the vertices $\{1,\ldots w\}$ then 
$$ e^0_1X_w=\hbox{min}_{\{1,\ldots, w\}\subseteq Y\subseteq X}{f^0Y-w\over f^1Y}.$$
\end{definition}
This is similar to $e$ in BHK, but involves the ratio of vertices to edges rather than to $2$-faces.  

\begin{lemma} If $X$ is a 2-dimensional connected simplicial complex with $e^0_1X_w>{1\over 3}$ then
$$f^1X\leq {3\chi X-3w+LX\over 3e^0_1X_w-1}.$$
\end{lemma}
\begin{proof}
See the proof of BHK Lemma 5.1.
\end{proof}

\begin{lemma}
For every $e>{1\over 3}$ there is $\beta$ so that every connected $2$-complex with $e^0_1(X)>e$, $L(X)\leq 0$ and $\chi(X)\leq 1$ and any contractible loop $\gamma:C_r\rightarrow X$ satisfies $A(\gamma)<\beta L(\gamma)$.  
\end{lemma}
\begin{proof}  
See the proof of BHK 5.2.  In this case the bound on $f^1$ from Lemma 3.2 replaces that on $f^2$ to yield only finitely many complexes to check.  
\end{proof}

\begin{lemma}
For every $\epsilon>{1\over 3}$ there is some $\beta$ so that every minimal filling ($C_r\rightarrow D\rightarrow X$) with $e^0_1X\geq \epsilon$ and $\chi Z\leq1$ for every connected $Z\subseteq X$ has $$f^2D<\beta(r+f^2(D-D_{\leq 0})).$$  
\end{lemma}
\begin{proof}
See the proof of BHK 5.9 and the definition before 5.5 in BHK.
\end{proof}

\begin{lemma}If $\epsilon>{1\over 3}$ and ($C_r\rightarrow D\rightarrow X$) is a minimal filling with $e^0_1X\geq \epsilon$ then 
$$f^2(D-D_{\leq 0})\leq{8r\over 9e-3}.$$  
\end{lemma}
\begin{proof}
See the proof of BHK 5.10 and use $LX_{ij}^\pi=2f^1X_{ij}^\pi-3f^2X_{ij}^\pi\geq 1$ so that 
$(3e-1)f^1X_{ij}^\pi\geq {3\over 2}(3e-1)f^2X_{ij}^\pi$.  
\end{proof}

\begin{lemma}
For every $\epsilon>{1\over 3}$ there is $\lambda$ such that for every $X$ with $e^0_1X\geq\epsilon$, every contractible loop $\gamma:C\rightarrow X$ satisfies $$A\gamma\leq\lambda L\gamma.$$
\end{lemma}
\begin{proof}
See the proof of BHK 3.7.
\end{proof}
\begin{lemma}
If $X$ has an embedded cycle $\gamma:C_6\rightarrow X$ having $\gamma(2i)=i$
for every $i\in\{1,2,3\}$ and $e^0_1X_3>{1\over 3}$ then $\gamma$ is not contractible in $X$.  
\end{lemma}
\begin{proof}
(See the proof of BHK 3.13)  If $\gamma=(1,a,2,b,3,c)$ is a contractible cycle in 
$X$ then by the second part of theorem 2.1 there is $Z\subseteq X$ such that every 
connected $Z'\subseteq Z$ has $\chi Z'\leq 1$ and $\gamma$ is contractible in $Z$.  
Let ($C_6\rightarrow D\rightarrow Z)$ be a minimal filling of $\gamma$ in $Z$.  
By BHK Lemma 5.4, $\pi$ is a 1-immersion so that no images of interior edges 
contribute positively to $L\gamma$ and $$L(\hbox{Im}(\pi))\leq LD\leq 6.$$  
By Lemma 3.1 there is $$f^1(\hbox{Im}(\pi))\leq {3\chi(\hbox{Im}(\pi))-3\cdot 3 +L(\hbox{Im}(\pi))\over 3e^0_1(\hbox{Im}(\pi))_3-1}\leq{3-9+6\over\cdots}=0.$$
This is a contradiction and $\gamma$ is not contractible in $X$.  
\end{proof}
\begin{definition}(See BHK Definition 3.9)
A $2$-dimensional simplicial complex $X$ is $(\epsilon^0_1, m, r)$-sparse if every $2$-dimensional simplicial subcomplex $Z\subseteq X$ containing the vertices $\{1,\ldots ,r\}$ with $f^0Z\leq m$ satisfies $e^0_1Z_r<\epsilon$.
It is $(\epsilon^0_1, m, r)$-full if every such complex $Z$ occurs as a subcomplex of $X$.  
\end{definition}

\begin{lemma}
If $m$ and $r$ are positive integers, $\epsilon>0$ and every $p_n\leq n^-\epsilon$ then $K(n,p_n)$ is aas $(\epsilon^0_1, m, r)$-sparse, while if every $p_n\geq n^-\epsilon$ then $K(n,p_n)$ is aas $(\epsilon^0_1, m, r)$-full.  
\end{lemma}
\begin{proof}
See the proof of BHK 3.10 for the sparsity.  Full follows from an easy second moment argument.  
\end{proof}

\begin{lemma}
For every $\epsilon>{1\over 3}$ there are $m$ and $\rho$ such that every contractible loop $\gamma:C_r\rightarrow X$ in an $(\epsilon^0_1,m,0)$-sparse complex $X$ satisfies $A(\gamma)<\rho L(\gamma)$.
\end{lemma}
\begin{proof}
See the proof of the first part of Lemma 3.12 in BHK and use Lemma 3.5 in place of BHK Lemma 3.7. \end{proof}

\smallskip\noindent{\it Proof of Theorem 1.1:}
Since $\epsilon<{1\over 2}$ Lemma 3.9 with $r=3$ implies that $K(n,p_n)$ has aas a cycle $\gamma:C_6\rightarrow X$ with $\gamma(2i)=i$ for $i\in\{1,2,3\}$.  
By Lemma 3.6 $\gamma$ is aas not contractible in $X$.  
\hfill$\square$\medskip

I am assured by Matthew Kahle that the arguments his paper [4] give an aas spectral gap larger than ${1\over 2}$ for appropriate Laplacians at all vertex links of $K(n,p_n)$ if $p_n\geq n^{-{1\over 2}+\epsilon}$ and that this together with a Garland type argument of \.Zuk imply
\begin{theorem}
If $\epsilon>0$ and $n^{-{1\over 2}+\epsilon}\leq p_n\leq n^{-{1\over 3}-\epsilon}$ then $\pi_1(K(n,p_n))$ aas has Kazhdan's property T. 
\end{theorem}
\section{questions}
\noindent 
Write $K_4(n,p)$ for the measure on cell complexes given by adding a two cell to all possible cycles of length $3$ and of length $4$ in the Erd\H{o}s-R\'{e}nyi random graph $K(n,p)$ and $\pi_1(K_4(n,p))$ for the associated measure on groups.  

\noindent {\bf Question:} For which $\epsilon$ is $\pi_1(K_4(n,n^{-\epsilon}))$aas trivial?

For this question $4$-admissible webs appear to replace the $3$-admissible ones arising in the clique complexes, but the local reduction methods used here do not seem to work as easily.   

\bibliographystyle{amsplain}

\def\cprime{$'$} \def\cprime{$'$}
\providecommand{\bysame}{\leavevmode\hbox to3em{\hrulefill}\thinspace}
\providecommand{\MR}{\relax\ifhmode\unskip\space\fi MR }
\providecommand{\MRhref}[2]{%
  \href{http://www.ams.org/mathscinet-getitem?mr=#1}{#2}
}
\providecommand{\href}[2]{#2}

\vskip20pt

\noindent {\bf Author:} Eric Babson

\end{document}